\def\draft{n}
\documentclass{amsart}
\usepackage[headings]{fullpage}
\usepackage{amsfonts,amsmath,amstext,amsbsy}
\usepackage{euscript,amssymb,graphicx,color}

\theoremstyle{plain}

\newtheorem{theorem}{Theorem}[section]
\newtheorem{proposition}{Proposition}[section]
\newtheorem{lemma}[proposition]{Lemma}

\theoremstyle{definition}
\newtheorem{definition}[proposition]{Definition}

\theoremstyle{remark}
\newtheorem{example}[proposition]{Example}

\newtheorem{remark}[proposition]{Remark}

\def\printname#1{
        \if\draft y
                \smash{\makebox[0pt]{\hspace{-0.5in}
                        \raisebox{8pt}{\tt\tiny #1}}}
        \fi
}

\newlength{\standardunitlength}
\catcode`\@=11
\long\def\@makecaption#1#2{%
     \vskip 10pt

\setbox\@tempboxa\hbox{
       \small\sf{\bfcaptionfont #1. }\ignorespaces #2}%
     \ifdim \wd\@tempboxa >\captionwidth {%
         \rightskip=\@captionmargin\leftskip=\@captionmargin
         \unhbox\@tempboxa\par}%
       \else
         \hbox to\hsize{\hfil\box\@tempboxa\hfil}%
     \fi}
\font\bfcaptionfont=cmssbx10 scaled \magstephalf
\newdimen\@captionmargin\@captionmargin=2\parindent
\newdimen\captionwidth\captionwidth=\hsize
\catcode`\@=12

\def\lbl#1{\label{#1}\printname{#1}}

\def\BN{\mathbb N}
\def\BZ{\mathbb Z}

\def\BQ{\mathbb Q}

\def\e{\epsilon}

\def\ef{\equiv_f}

\makeatletter
 \def\LaTeX{\leavevmode L\raise.42ex
   \hbox{\kern-.3em\size{\sf@size}{0pt}\selectfont A}\kern-.15em\TeX}
\makeatother

\newcommand{\BibTeX}{{\rm B\kern-.05em{\sci\kern-.025emb}\kern-.08em\TeX}}

\newcommand{\modo}[2]{#1 \equiv_f #2}

\newcommand{\lcm}{\mathrm{lcm}}

\newcommand{\Rec}[1]{\mathrm{\bf REC}_{\bf #1}}

\definecolor{red}{rgb}{1.000,0.000,0.000}
\definecolor{green}{rgb}{0.000,0.000,1.000}
\definecolor{blue}{rgb}{0.000,0.000,1.000}
\definecolor{orange}{rgb}{0.800,0.600,0.000}
\definecolor{cyan}{rgb}{0.000,0.600,0.600}
\definecolor{magenta}{rgb}{1.000,0.000,1.000}

\begin{document}


\title[A new algorithm for the recursion of multisums with 
improved universal denominator]
{A new algorithm for the recursion of hypergeometric multisums with 
improved universal denominator}

\author{Stavros Garoufalidis}
\address{School of Mathematics \\
         Georgia Institute of Technology \\
         Atlanta, GA 30332-0160, USA \\ 
         {\tt http://www.math.gatech} \newline {\tt .edu/$\sim$stavros } }
\email{stavros@math.gatech.edu}
\author{Xinyu Sun}
\address{Department of Mathematics \\
        Tulane University \\
        6823 St. Charles Ave \\
        New Orleans, LA 70118, USA \\
{\tt http://} \newline {\tt www.math.tulane.edu/$\sim$xsun1}}
\email{xsun1@tulane.edu}

\thanks{S.G. was supported in part by National Science 
Foundation. \\
\newline
{\em Mathematics Classification.} Primary 33F10. Secondary 05E99.
}

\date{July 8, 2009}


\begin{abstract}
The purpose of the paper is to introduce two new algorithms. The first one
computes a linear recursion for proper hypergeometric multisums, by treating
one summation variable at a time, and provides rational certificates along 
the way. A key part in the search of a linear recursion is an improved 
universal denominator algorithm that constructs all rational solutions
$x(n)$ of the equation
$$
\frac{a_m(n)}{b_m(n)}x(n+m)+\cdots+\frac{a_0(n)}{b_0(n)}x(n)=
c(n),$$
where $a_i(n), b_i(n), c(n)$ are polynomials. 
Our algorithm improves Abramov's universal denominator.

\end{abstract}

\maketitle

\tableofcontents


\section{Introduction}
\lbl{sec.intro}

\subsection{History}
\lbl{sub.history}

The paper introduces a new algorithm to find linear recursions 
(with coefficients polynomials in $n$) for multidimensional sums of the form

\begin{equation}
\lbl{eq.Sn}
S(n)=\sum_{k \in D} f(n,k),
\end{equation}
where $D \subset \BZ^r$ and the summand $f(n,k)$ is a 
{\em proper hypergeometric term} in the variables $(n,k)$. By proper 
hypergeometric term (abbreviated by {\em term}) $f(m)$ in the variables 
$m=(m_1,\dots,m_s)$ we mean an expression of the form  

\begin{equation}
\lbl{eq.fnk}
f(m)=P(m)\prod_{j=1}^J A_j(m)!^{\e_j}
\end{equation}
where $P(m)$ is a polynomial in $m$ and 
$A_j(m)=\sum_{i=1}^s a_{ji} m_j$ is a linear form in $m$ with integer
coefficients $a_{ji}$ and $\e_j =\pm 1$ for $1 \leq j \leq J$. Throughout
this paper, $f(n,k)$ will denote a proper hypergeometric term.

As observed by Zeilberger \cite{Ze}, and further explained in \cite{WZ}, 
Sister Celine's 
method \cite{Fas} can be used to prove the existence of linear recursions of 
$S(n)$ in a constructive way. A faster algorithm was constructed by 
Zeilberger (also known as {\em creative telescoping} \cite{PWZ}), which 
employed Gosper's indefinite summation algorithm \cite{Gos}. Creative
telescoping is faster than Sister Celine's method, and often returns 
the optimal (i.e., minimal order) recursions. However, due to the nature of 
Gosper's algorithm, Zeilberger's method only works for single sums, i.e.,
when $r=1$ in \eqref{eq.Sn}.

Wegschaider in \cite{Weg} improved Sister Celine's algorithm for 
multisums; Zeilberger has a program {\tt EKHAD} for creative telescoping, 
while Paule and Schorn \cite{PSh} implemented it in {\tt Mathematica};
Schneider created a package called {\tt Sigma}, the framework of which 
was explained in \cite{Sch}; Apagodu and Zeilberger \cite{AZ} generalized 
creative telescoping to multi-variable context which resulted in another fast 
algorithm.

\subsection{What is multivariable creative telescoping?}
\lbl{sub.telescoping}

{\em Multivariable creative telescoping} for $S(n)$ is the problem of
finding a natural number $J \in \BN$, and rational functions 
$a_j(n) \in \BQ(n)$, for $1 \leq j \leq J$ and rational functions
$C_i(n,k) \in \BQ(n,k)$ for $1 \leq i \leq r$ so that
\begin{equation}
\lbl{eq.Cert}
\sum_{j=0}^J a_j(n) N^j f(n,k)=
\sum_{i=1}^r (K_i-1) (C_i(n,k) f(n,k)),
\end{equation}  
where $N,K_i,n,k_i$ are operators that act on functions $f(n,k)$ as follows: 
\begin{eqnarray}
(Nf)(n,k)&=&f(n+1,k), \qquad (nf)(n,k)=nf(n,k),	\nonumber \\ 
(K_i f)(n,k_1,\ldots,k_r)&=&
f(n,k_1,\ldots,k_{i-1}, k_i+1, k_{i+1}, \ldots, k_r), \\ 
(k_i f)(n,k_1,\ldots,k_r)&=&k_i f(n,k_1,\ldots,k_r). \nonumber 
\end{eqnarray}
Note that the operators $N,n,K_i,k_i$ commute except in the 
following instance
\begin{equation}
Nn=n+1, \qquad K_ik_i=k_i+1.
\end{equation}
In Equation~\eqref{eq.Cert}, the rational functions $C_i(n,k)$ for 
$1 \leq i \leq r$ are called the {\em certificates} and the operator
$\sum_{j=0}^J a_j(n)N^j$ is called the {\em recursion} for the sum $S(n)$. 
Given Equation \eqref{eq.Cert}, we can sum over $k$ to obtain an inhomogeneous 
linear recursion for $S(n)$, whose inhomogeneous part consists of the 
contribution from the boundary terms. 

All known algorithms of creative telescoping convert 
\eqref{eq.Cert} to a system of {\em linear equations} with coefficients
in the field $\BQ(n,k)$. This is possible since
dividing both sides of \eqref{eq.Cert} by $f(n,k)$ and using the fact that 
$f(n,k)$ is proper hypergeometric, it follows that the ratios
$Nf(n,k)/f(n,k) \in \BQ(n,k)$ and $K_if(n,k)/f(n,k) \in \BQ(n,k)$ are
rational functions. The number of unknowns and equations 
directly affect the performance of the above mentioned algorithms.

\subsection{Abramov's universal denominator algorithm}
\lbl{sub.abramov}

A key part of our search for a linear recursion of hypergeometric multisums
is an improved universal denominator algorithm that finds all rational
solutions $x(n) \in \BQ(n)$ to a linear difference 
equation
$$
\frac{a_m(n)}{b_m(n)}x(n+m)+\cdots+\frac{a_0(n)}{b_0(n)}x(n)=
c(n),$$
where $a_i(n), b_i(n), c(n)$ are polynomials. The idea is to correctly 
predict the denominator $u(x)$ of $x(n)$ (also known as {\em the 
universal denominator}), so that the problem can be reduced into finding a 
polynomial solution to a linear difference equation. In \cite{Ab} Abramov 
developed a universal denominator algorithm.
In this paper, we develop a new algorithm that improves Abramov's algorithm by
possibly reducing the number of factors in the universal denominator.
The new algorithm is used repeatedly to convert the problem of finding 
recursions of multivariate hypergeometric sums into the problem of solving 
system of linear equations. And fewer factors in the universal denominator 
implies fewer numbers of variables and fewer equations in the system.

\subsection{Acknowledgement}
The authors wish to thank D. Zeilberger for a careful reading of an earlier 
version of the paper and for detailed suggestions and comments.

\section{Two algorithms}
\lbl{sec.algorithm}

\subsection{A new algorithm for the recursion of hypergeometric multisums}
\lbl{sub.alg}

To describe our algorithm for the recursion of multisums, let us introduce 
some useful notation.

\begin{definition}
\lbl{def.ef}
Fix a term $f(n,k)$ where $k=(k_1,\dots,k_r)$ and $1 \leq i, j \leq r$.
We say that two operators $P$ and $Q$ in the variables $n,k_i,N$ and $K_i$
are $f$-equivalent modulo $K_i-1,\dots,K_j-1$, and write
\begin{equation}
\lbl{eq.ef}
P \ef Q \bmod (K_i-1,K_{i+1}-1,\dots,K_j-1),
\end{equation}
if there exist rational functions $b_s(n,k)$ for $i \leq s \leq j$
so that
\begin{equation}
\lbl{eq.ef2}
(P-Q)f(n,k)=\sum_{s=i}^j (K_s-1)(b_s(n,k)f(n,k)).
\end{equation}
If $i > j$, the right-hand side of the last equation is 0.
\end{definition}
%
Our algorithm will construct operators $\Rec{i}$ for $0 \leq i \leq r$ and
$\Rec{j,i}$ for $0 \leq i < j \leq r$ of the following form

\begin{align*}
\Rec{i} &= \sum_{s=0}^{d_i} a_{i,s}(n,k_1,\dots,k_i) K_i^s,  \quad i \neq 0; 
&
\Rec{0} &= \sum_{s=0}^{d_0} a_{0,s}(n) N^s;   \\
\Rec{j,i} &= K_i + \sum_{s=0}^{d_{j,i}} a_{j,i,s}(n,k_1,\dots,k_j) K_j^s,
\quad i \neq 0; & 
\Rec{j,0} &= N + \sum_{s=0}^{d_{j,0}} a_{j,0,s}(n,k_1,\dots,k_j) K_j^s,
\end{align*}
that satisfy

\begin{equation}
\lbl{eq.Cri}
\Rec{i}  \ef  0 \bmod (K_{i+1}-1,\dots,K_r-1) \qquad
\Rec{j,i}  \ef  0 \bmod(K_{j+1}-1,\dots,K_r-1).
\end{equation}
We will call such operators $\Rec{i}, \Rec{j,i}$ $f$-{\em compatible}.

Observe that $\Rec{0}$ is exactly Equation \eqref{eq.Cert}.
Here are the steps for the algorithm.

\begin{center}
\fbox{
\begin{tabular}{ll}
INPUT:	& A proper hypergeometric term $f(n, k_1, \ldots, k_r)$.
\\
OUTPUT:		& 
A recursion $\sum_{i=0}^I a_i(n)N^i$ and a set of certificates 
$C_i(n, k_1, \ldots, k_r)$ that satisfy \eqref{eq.Cert}. 
\\ \hline
&
\\
Step 1.		& Set $l := r$, $k_0=n$ and $K_0 := N$;		
\\ 
Step 2.		& Set $\Rec{r,0} := N - \frac{Nf}{f}$, 
$\Rec{r,i} := K_i - \frac{K_if}{f}$, $1 \le i \le r-1$ 
\\ 
& and $\Rec{r}:=
K_r - \frac{K_rf}{f}$;	
\\
Step 3.	
& Construct $\Rec{r-1}$ using Proposition \ref{prop.2var}.
\\
Step 4.		& If $l = 1$, print $\Rec{0}$ and stop; otherwise, continue;
\\
Step 5.		& Construct $\Rec{l-1,i}$ for $0 \leq i \leq l-2$
using Proposition \ref{prop.4var}.
\\		
Step 6.		& Construct $\Rec{l-2}$ using Proposition \ref{prop.2var}.
\\		
Step 7.		& Set $l=l-1$, and go to Step 4.
\end{tabular}
}
\end{center}
There is some similarity between our algorithm and results of Schneider
\cite{Sch}; we do believe however the underlying algorithm to obtain 
the certificates is different from Schneider's program {\tt Sigma},
although he did employ some version of Abramov's algorithm.

The subtle part of the above algorithm are steps 5 and 6 which compute 
the proper denominators for the certificates that appear in 
Equations \eqref{eq.Cri}. This is done using Propositions \ref{prop.4var}
and \ref{prop.2var}, which follow from Theorem~\ref{thm.multigosper}, 
which are implemented in our improved
denominator algorithm of Section \ref{sub.denominators}.

\begin{example}
\lbl{ex.schem}
When $r=3$ the algorithm computes $\Rec{i}$ for $0 \leq i \leq 3$ and
$\Rec{i,j}$ for $0 \leq j <i \leq 3$ in the following order:

$$
\Rec{3,0}, \Rec{3,1},\Rec{3,2},\Rec{3} \to \Rec{2} \to \Rec{2,0}, 
\Rec{2,1} \to \Rec{1} \to \Rec{1,0} \to \Rec{0}
$$
\end{example}
A {\tt Maple} implementation of the above algorithm is available at \cite{GS2}.
A {\tt Mathematica} implementation will be developed later. A $q$-version
of the above algorithm is possible and will also be developed later.

\subsection{An improved universal denominator algorithm}
\lbl{sub.denominators}

In this section we present our universal denominator algorithm. Let $K$
denote a fixed field, which in applications it is the field of rational
functions with rational coefficients in a finite set of variables.

\begin{center}
\fbox{
\begin{tabular}{ll}
INPUT:	& An equation with rational coefficients
$\frac{a_m(s)}{b_m(s)}x(s+m)+\cdots+\frac{a_0(s)}{b_0(s)}x(s)
=c(s)$, 
\\
		& where $a_i(s), b_i(s), c(s) \in K[s]$ are polynomials. 
\\
OUTPUT:	& 
A  rational solution  $x(s)=\frac{R(s)P(s)}{Q(s)}$ to the equation, where 
$P(s), Q(s),R(s) \in K[s]$. 	
\\ \hline
&
\\
Step 1.		& Set $P(s)=1$;
\\ 
Step 2.		& Set $\sigma(s)=\lcm\left(b_0(s),\ldots,b_m(s)\right)$, 
and $\tau_i(s)=a_i(s-i)\frac{\sigma(s-i)}{b_i(s-i)}$, $0 \le i \le m$;
\\ 
Step 3.		& Set $Q(s) = \gcd(\tau_0(s),\ldots,\tau_m(s))$;
\\
Step 4.		& Find the largest possible nonnegative integer $\ell$ 
such that 
$\gcd(\frac{\tau_0(s)}{Q(s)},\frac{\tau_m(s-\ell)}{Q(s)}) = \phi(s) \neq 1$; 
\\
Step 5.		& If such an $\ell$ does not exists, continue to Step 9;
\\		
Step 6.		& Otherwise, set $Q(s)=Q(s) \prod_{i=0}^{\ell}\phi(s+i)$;
\\		
Step 7.		& Set $\tau_0(s) = \frac{\tau_0(s)}{\phi(s)}$ and 
$\tau_m(s) = \frac{\tau_m(s)}{\phi(s+\ell)}$;
\\
Step 8.		& Go to Step 4;
\\
Step 9.		& Set 
$R(s)=\lcm_{0 \le i \le m}\left\{\frac{b_i(s-i)Q(s)}{\gcd\left(b_i(s-i)Q(s),
a_i(s-i)d(s-i)P(s)\prod_{j \neq i}b_j(s-i)Q(s+j-i) \right)}\right\}$;
\\
Step 10.	& If $R(s)=1$, STOP;
\\
Step 11.	& Otherwise, set $P(s)=P(s)R(s)$;
\\
Step 12.	& Go to Step 9.
\\
\end{tabular}
}
\end{center}

\subsection{Plan of the proof}
\lbl{sub.plan}

The structure of the paper is as follows. In Section~\ref{sec.example}, 
we explain the usage of the {\tt Maple} program. We present a few examples and 
compare the results and performance against the programs discussed above.
In Section~\ref{sec.algorithm}, we introduce the terminology used in 
the paper, and present the general structure of the method as a sequence
of steps.
In Section~\ref{sec.proof}, we prove the validity of each step of the 
structure, and also explain the method in detail. 
In Section~\ref{sec.denominator}, we prove a new algorithm 
that generates universal denominators with possibly less factors than those 
generated by Abramov's algorithm, 
that also partially predict the numerators for rational solutions 
to linear difference equations.

\section{Use of the program and examples} 
\lbl{sec.example}

\begin{example}
\lbl{ex.1}
Define 
\begin{eqnarray*}
f(n,k_1,k_2) &=& (-1)^{n+k_1+k_2}\binom{n}{k_1}\binom{n}{k_2}\binom{n+k_1}{k_1}
\binom{n+k_2}{k_2}\binom{2n-k_1-k_2}{n}
\\
g(n,k) &=& \binom{n}{k}^4.
\end{eqnarray*}
We will prove that (\cite[Page 33]{PWZ} and \cite{Sch})
$$
\sum_{k_1,k_2}f(n,k_1,k_2) = \sum_k g(n,k).
$$
After running the program, both sides of the above equation are 
annihilated by the operator
$$
(n+2)^3 N^2-2(2n+3)(3n^2+9n+7) N-4(4n+5)(4n+3)(n+1).
$$
Since they have the same initial conditions for $n=0,1$, 
the two sides agree for all natural numbers $n$.

Please see \cite{GS2} for the syntax of input and output.
\end{example}

\begin{example}
\lbl{ex.2}
Define 
$$
f(n,k_1,k_2) = \binom{n}{k_1}\binom{n}{k_2}\binom{n+k_1}{k_1}
\binom{n+k_2}{k_2}\binom{2n-k_1-k_2}{n}.
$$
Please see \cite{GS2} for complete information. The recursion for the
multisum $\sum_{k_1,k_2} f(n,k_1,k_2)$ is of degree 4. 
\end{example}

\begin{example} 
\lbl{ex.3}
Define
\begin{equation*}
f(n,k_1,k_2,k_3)=
(-1)^{n+k_1+k_2+k_3}\binom{n}{k_1}\binom{n}{k_2}\binom{n}{k_3}
\binom{n+k_1}{k_1}\binom{n+k_2}{k_2}
\binom{n+k_3}{k_3}\binom{2n-k_1-k_2-k_3}{n}.
\end{equation*}
Please see \cite{GS2} for complete information. The recursion for the
multisum $\sum_{k_1,k_2,k_3} f(n,k_1,k_2,k_3)$ is of degree 4.
\end{example}


\section{Proof of the multisum algorithm} 
\lbl{sec.proof}

\subsection{Two Lemmas}
\lbl{sub.lemmas}

We fix a term $f(n,k)$ where $k=(k_1,\dots,k_r)$,
and consider a fixed variable $k_v$ 
and the corresponding operator $K_{v}$. The moduli are always 
$(k_{v+1}, \dots, k_r)$, which we suppress for simplicity.

\begin{lemma} 
\lbl{lem.reduction}
If $\sum^{I}_{i=0}b_i(n,k_1,\ldots,k_v)K_v^i \ef 0$ and 
$N + \sum^{I-1}_{i=0}a_i(n,k_1,\ldots,k_v)K_v^i \ef 0$, 
then for any integer $m$ and 
rational functions $\{\alpha_i(n,k_1,\ldots,k_v)\}_{0 \le i \le m}$, 
there exist
rational functions $\{\beta_j(n,k_1,\ldots,k_v)\}_{0 \le j \le I-1}$ so that
\begin{equation}
\lbl{eq.bjnr}
\sum^{m}_{i=0}\alpha_i(n,k_1,\ldots,k_v)N^i \ef 
\sum^{I-1}_{j=0}\beta_j(n,k_1,\ldots,k_v)K_v^j.
\end{equation}
Furthermore, Equation \eqref{eq.bjnr} is a linear system of equations
with unknowns $\{\beta_j(n,k_1,\ldots,k_v)\}$ and coefficients in the 
field $\BQ(n, k_1,\ldots,k_v)$.
\end{lemma}

\begin{proof}
Since the operators are linear over the field $\BQ(n, k_1,\ldots,k_v)$, 
we only need to 
show the result for $N^m$ for $m \ge 1$ by induction.

The conclusion is true for $m=1$. Suppose it is true for $m-1$, i.e.,
$N^{m-1} \ef \sum^{I-1}_{j=0}\gamma_j(n,k_1,\ldots,k_v)K_v^j$ for some 
$\{\gamma_i\}$. Then we have
\begin{eqnarray*}
N^m & \ef & N(N^{m-1}) \ef  N(\sum^{I-1}_{j=0}\gamma_j(n,k_1,\ldots,k_v)K_v^j) 
\ef \sum^{I-1}_{j=0}\gamma_j(n+1,k_1,\ldots,k_v)K_v^jN  \\
& \ef & \sum^{I-1}_{j=0}\gamma_j(n+1,k_1,\ldots,k_v)K_v^j
\left(-\sum^{I-1}_{i=0}a_i(n,k_1,\ldots,k_v)K_v^i\right) \ef 
\sum^{I-1}_{i=0}\beta_i(n,k_1,\ldots,k_v)K_v^i,
\end{eqnarray*}
for some rational $\beta_i(n,k_1,\ldots,k_v)$. The last equation is because 
the order of recursion in $k_v$ satisfied by $f$ is at most $I$.

Since the reduction of $N^m$ does not depend on 
$\{\alpha_i(n,k_1,\ldots,k_v)\}$, it 
follows that $\{\beta_j(n,k_1,\ldots,k_v)\}$ are linear functions of 
$\{\alpha_i(n,k_1,\ldots,k_v)\}$.
\end{proof}

\begin{lemma}
\lbl{lem.sum}
Given $K_v^p + \sum^{p-1}_{i=0}a_i(n, k_1,\ldots,k_v)K_v^i \ef 0$ 
and 
$$
-\sum_{j=0}^{p-1}a_{p-1-j}(n, k_1,\ldots,k_v+j)
b_{p-1}(n, k_1,\ldots,k_v+1+j)-b_{p-1}(n, k_1,\ldots,k_v) 
= \sum_{j=0}^{p-1}c_{p-1-j}(n, k_1,\ldots,k_v+j),
$$
where $\{c_i(n, k_1,\ldots,k_v)\}_{0 \le i \le p-1}$ and 
$b_{p-1}(n, k_1,\ldots,k_v)$ are rational functions.
Define, for $0 \le i < p-1$,
\begin{eqnarray*}
b_i(n, k_1,\ldots,k_v) 
&=& b_{p-1}(n, k_1,\ldots,k_v-p+1+i) 
    + \sum_{j=1}^{p-i-1}a_{i+j}(n, k_1,\ldots,k_v-j)
b_{p-1}(n, k_1,\ldots,k_v-j+1) \\
& & + \sum_{j=1}^{p-i-1}c_{i+j}(n, k_1,\ldots,k_v-j).
\end{eqnarray*}
Then
$$
\sum^{p-1}_{i=0}c_i(n, k_1,\ldots,k_v)k_v^i \ef 
(k_v-1)\sum^{p-1}_{i=0}b_i(n, k_1,\ldots,k_v)k_v^i.
$$
\end{lemma}

\begin{proof}
From the definition of $b_i(n, k_1,\ldots,k_v)$, it is easy to check that

\begin{eqnarray*}
b_{i-1}(n, k_1,\ldots,k_v+1) &=& b_i(n, k_1,\ldots,k_v) 
+ a_i(n, k_1,\ldots,k_v)b_{p-1}(n, k_1,\ldots,k_v+1)+c_i(n, k_1,\ldots,k_v),
\\
b_0(n, k_1,\ldots,k_v) &=& -b_{p-1}(n, k_1,\ldots,k_v+1)
a_0(n, k_1,\ldots,k_v)-c_0(n, k_1,\ldots,k_v).
\end{eqnarray*}
It follows that

\begin{eqnarray*}
(K_v-1)\sum^{p-1}_{i=0}b_i(n, k_1,\ldots,k_v)K_v^i	
& \ef & b_{p-1}(n, k_1,\ldots,k_v+1)K_v^p 
+ \sum^{p-2}_{i=0}b_i(n, k_1,\ldots,k_v+1)K_v^{i+1} 
\\
&     & - \sum^{p-1}_{i=0}b_i(n, k_1,\ldots,k_v)K_v^i	
\\	
& \ef & -\sum_{i=0}^{p-1}b_{p-1}(n, k_1,\ldots,k_v+1)
a_i(n, k_1,\ldots,k_v)K_v^i 
\\
&     & + \sum_{i=1}^{p-1}\left(b_{i-1}(n, k_1,\ldots,k_v+1)
-b_i(n, k_1,\ldots,k_v)\right)K_v^i - b_0(n, k_1,\ldots,k_v)	
\\	
& \ef & -\sum_{i=0}^{p-1}b_{p-1}(n, k_1,\ldots,k_v+1)
a_i(n, k_1,\ldots,k_v)K_v^i 
\\
&     & + \sum_{i=1}^{p-1}\left(a_i(n, k_1,\ldots,k_v)
b_{p-1}(n, k_1,\ldots,k_v+1)+c_i(n, k_1,\ldots,k_v)\right)K_v^i
\\
&        & + b_{p-1}(n, k_1,\ldots,k_v+1)a_0(n, k_1,\ldots,k_v)
+c_0(n, k_1,\ldots,k_v)
\\
& \ef & \sum^{p-1}_{i=0}c_i(n, k_1,\ldots,k_v)K_v^i.
\end{eqnarray*}
\end{proof}

Lemma~\ref{lem.sum} also appeared in \cite{Sch} in a different form. It is 
included here for completeness of the proofs.

\subsection{Two propositions for the algorithm}
\lbl{sub.keyprop}
In this section we state and prove Propositions \ref{prop.4var} and 
\ref{prop.2var} which are used in our algorithm.
Fix a term $f(n,k)$ where $k=(k_1,\dots,k_r)$. Recall we set $k_0=n$ and 
$K_0=N$.

\begin{proposition} 
\lbl{prop.4var}
Let $1 \le v < r$. Given $f$-compatible operators 
$\Rec{v+1}, \Rec{v}, \Rec{v+1,u}, \Rec{v+1,v}$
for $0 \leq u \le v \leq r$, it is possible to construct
an $f$-compatible operator $\Rec{v,u}$ for $0 \leq u < v$ in Step 5.
\end{proposition}

\begin{proposition}
\lbl{prop.2var}
Let $1 \le v \le r$. Given $f$-compatible operators $\Rec{v},\Rec{v,v-1}$ for 
$0 \le u \le v$, it is possible to construct $f$-compatible operator
$\Rec{v-1}$ in Steps 3 and 6. 
\end{proposition}

\subsection{Proof of Proposition \ref{prop.4var}}
\lbl{sub.4var}

Let
\begin{equation}
\lbl{eq.recv1}
\Rec{v+1} := 
K_{v+1}^J+\sum_{i=0}^{J-1}a_i(k_0,k_1, \ldots, 
k_{v+1})K_{v+1}^i \qquad \text{and}
\qquad \Rec{v+1} \ef 0 \bmod (K_{v+2}-1,\ldots,K_r-1).
\end{equation}
We can always divide the operator by the leading coefficient if it is not 1,
since it does not involve variables $k_{v+2},\ldots,k_{r}$.
Let us look for 
\begin{eqnarray} 
\lbl{eqn.exist}
\Rec{v,u} := K_u + \sum_{i=0}^{I-1}
\phi_i(k_0,k_1, \ldots, k_{v})K_{v}^i \qquad \text{and} \qquad
\Rec{v,u} \ef 0 \bmod(K_{v+1}-1,\ldots,K_r-1),
\end{eqnarray}
for some rational functions $\phi_i(k_0,k_1, \ldots, k_{v})$.
To prove the existence of $\Rec{v,u}$, borrow the idea in the proof 
of \cite{WZ} by solving
\begin{eqnarray} 
\lbl{eqn.exist.celine}
\left[K_u+\sum_{i_v=0}^{U_v}\cdots\sum_{i_m=0}^{U_m}
	\left(\sigma_{i_v \cdots i_m}(k_0,k_1, \ldots, k_{v})
\prod_{l=v}^{m}K_l^{i_l}\right)\right]f(k_0,k_1, \ldots, k_{r})=0,
\end{eqnarray}
with $\sigma_{i_v \cdots i_m}(k_0,k_1, \ldots, k_{v})$ being the unknown 
rational functions. Divide both sides by the hypergeometric function 
$f(k_0,\dots,k_{r})$ to obtain an equation of rational functions. By 
comparing the coefficients of the 
powers of $k_{v+1},\ldots,k_{r}$, we can set up 
a system of linear equations over the field $\BQ(k_0,k_1, \ldots, k_{v})$, 
whose unknowns are $\sigma_{i_v \cdots i_m}(k_0,k_1, \ldots, k_{v})$.
The number of unknowns is
$\prod_{l=v}^{m}(U_l+1)$, while the number of equations, 
which equals the degree of the numerator in 
Equation~\eqref{eqn.exist.celine}, is proportional to
$\left(\prod_{l=v}^{m}U_l\right)\left(\sum_{l=v}^{m}\frac{1}{U_l}\right)$.
It follows that when $U_v, \ldots, U_m$ are large enough, we have more 
unknowns than equations in the system, which guarantees a nontrivial solution. 
Replacing $K_{v+1}, \ldots, K_{r}$ in Equation~\eqref{eqn.exist.celine} 
with 1, we get a solution to Equation~\eqref{eqn.exist}. 
The maximum power of $I-1$ on $K_v$ is ensured by the existence of a 
recursion of order $I$.
The readers may also compare with 
\cite[Theorem 4.4.1]{PWZ} or \cite[Theorem MZ]{AZ} for a detailed 
discussion on the method in similar cases. With the proof of existence 
completed, we can introduce a new method to find the functions 
$\{\phi_i\}$ and $\{b_j\}$.

Reduce $K_u + \sum_{i=0}^{I-1} \phi_i(k_0,k_1, \ldots, k_{v})K_v^i$ into 
$\sum_{i=0}^{J-1}c_i(k_0,k_1, \ldots, k_{v+1})K_{v+1}^i$ 
for some rational \newline
$c_i(k_0,k_1, \ldots, k_{v+1})$, using  
Lemma~\ref{lem.reduction} below. This implies that 
$$
\modo{\sum_{i=0}^{J-1}c_i(k_0,k_1, \ldots, 
k_{v+1})K_{v+1}^i}{(K_{v+1} - 1)
\left(\sum_{i=0}^{J-1}{b_i(k_0,k_1, \ldots, k_{v+1})K_{v+1}^i}\right)}.
$$ 
Since the coefficient of $K_{v+1}^J$ is 1 in \eqref{eq.recv1}, it follows from 
Lemma~\ref{lem.sum} below 
that we only need to find $b_{J-1}(k_0,k_1, \ldots, k_{v+1})$ such that
\begin{eqnarray} 
\lbl{eqn.main.1}
-\sum_{j=-1}^{J-1}a_{J-1-j}(k_0,k_1, \ldots, 
k_{v+1}+j)b_{J-1}(k_0,k_1, \ldots, k_{v+1}+1+j) \nonumber 
\\	= \sum_{j=0}^{J-1}c_{J-1-j}(k_0,k_1, \ldots, k_{v+1}+j).
\end{eqnarray}
In the equation, $\{a_i(k_0,k_1, \ldots, k_{v+1})\}_{0 \le i \le J-1}$ 
are known; 
$b_{J-1}$ is a rational function over $k_0,\ldots,k_{v+1}$ over the field 
$\BQ(k_0,k_1, \ldots, k_{v})$; and $\{c_i\}_{0 \le i \le J-1}$ are linear 
over $\{\phi_j\}_{0 \le j \le I-1}$. So the right-hand side can be 
written as
$\frac{\sum_{j=0}^{I-1}U_{j}(k_0,k_1, \ldots, 
k_{v+1})\phi_j(k_0,k_1, \ldots, k_{v})}{V(k_0,k_1, \ldots, k_{v+1})}$, 
with polynomials $\phi_j(k_0,k_1, \ldots, k_{v})$ unknown;  
and $U_{j}(k_0,k_1, \ldots, k_{v+1})$ and $V(k_0,k_1, \ldots, k_{v+1})$ known.
Multiply both sides of 
Equation~\eqref{eqn.main.1} by $V(k_0,k_1, \ldots, k_{v+1})$ to obtain
\begin{eqnarray} 
\lbl{eqn.maineq}
\sum_{j=-1}^{J-1}-a_{J-1-j}(k_0,k_1, \ldots, 
k_{v+1}+j)b_{J-1}(k_0,k_1, \ldots, k_{v+1}+1+j)V(k_0,k_1, \ldots, k_{v+1}) 
\nonumber 
\\
= \sum_{j=0}^{I-1}U_{j}(k_0,k_1, \ldots, k_{v+1})
\phi_j(k_0,k_1, \ldots, k_{v}). 
\end{eqnarray}
In the above equation, consider $b_{J-1}(k_0,\dots,k_{v+1}) \in 
\BQ(k_0,\dots,k_v)(k_{v+1})$, and apply Theorem~\ref{thm.multigosper} 
to the field $K=\BQ(k_0,\dots,k_v)$ and the variable
$s=k_{v+1}$. It follows that we can write
$$
b_{J-1}(k_0,k_1, \ldots, k_{v+1}) = 
\frac{R(k_0,k_1, \ldots, k_{v+1})P(k_0,k_1, 
\ldots, k_{v+1})}{Q(k_0,k_1, \ldots, k_{v+1})},
$$
with polynomials $R(k_0,k_1, \ldots, k_{v+1}) \in \BQ[k_0,\dots,k_{v+1}]$ 
and 
$Q(k_0,k_1, \ldots, k_{v+1}) \in \BQ[k_0,\dots,k_{v+1}]$ known, and 
$P(k_0,k_1, \ldots, k_{v+1}) \in \BQ[k_0,\dots,k_{v+1}]$ unknown. 
By multiplying both sides by the 
common denominator of the left-hand side, and comparing the degree of 
$k_{v+1}$, 
we can determine the degree of $k_{v+1}$ in $P(k_0,k_1, \ldots, k_{v+1})$, 
say, $L$. By writing  $P(k_0,k_1, \ldots, k_{v+1})$ as 
$\sum_{i=0}^{L}\psi_i(k_0,k_1, \ldots, k_{v})k_{v+1}^i$, plugging it back 
into Equation~\eqref{eqn.maineq}, and comparing the coefficients of powers of 
$k_{v+1}$, we can set up a system of linear equations with 
$\{\phi_j\}_{0 \le j \le I-1}$  and $\{\psi_i\}_{0 \le i \le L}$ as unknowns.
The system is guaranteed to have a nontrivial solution because of the 
existence of the recursion.
\qed

\subsection{Proof of Proposition \ref{prop.2var}}
\lbl{sub.2var}

The existence of the recursion can be proved in a way similar to 
Theorem~\ref{prop.4var}. And the method of the new algorithm is 
also the same. Basically we again rewrite the left-hand side of the equations 
into powers of $K_v$, compare their coefficients on both sides, and solve the 
resulting linear equations. Details are omitted.
\qed

\section{Proof of the universal denominator algorithm}
\lbl{sec.denominator}

In this section we state and prove Theorem~\ref{thm.multigosper}
which determines the denominator and 
partially the numerator of the rational function $b_{J-1}$ 
in Equation~\eqref{eqn.maineq}. This is crucial for the performance of the 
algorithm as a whole, 
because it reduces the number of variables and number of equations in the 
final system 
of linear equations to be solved. The most straight-forward guess for the 
denominator $b_{J-1}$ in Equation~\eqref{eqn.maineq}, i.e., the denominator 
of the right-hand side of the equation, 
will give us an algorithm whose performance is compatible to that of Sister 
Celine's method 
on a single step. Theorem~\ref{thm.multigosper} also improves 
Abramov's universal denominator \cite{Ab}.

Let $K$ denote a field, which for our applications it will be the
field of rational functions with rational coefficients in a finite set
of variables. Let $s$ denote a fixed variable that does not appear in $K$.
As usual, if $p(s), q(s) \in K[s]$ are polynomials, then we write 
$p(s) \left| q(s) \right. $ if $p(s)$ divides $q(s)$.

Consider the equation
\begin{equation}
\lbl{eq.abx}
\sum_{i=0}^{m}\frac{a_i(s)}{b_i(s)}x(s+i) = c(s),
\end{equation}
where $a_i(s), b_i(s), c(s) \in K[s]$ are polynomials, and $\gcd(a_i, b_i)=1$.
Define 
\begin{eqnarray*}
\sigma(s)	&=&	\lcm\left(b_i(s) \left| 0 \le i \le m \right.\right),
\\
\tau_i(s)	&=& \frac{a_i(s-i)}{b_i(s-i)}\sigma(s-i), 0 \le i \le m, 
\\
\hat{\tau}(s)	&=& \gcd\left(\tau_0(s), \ldots, \tau_m(s)\right)
\end{eqnarray*} 
and
\begin{eqnarray}
\lbl{eq.Qs}
Q(s)	&=& \hat{\tau}(s)\prod_{i=0}^{I} \prod_{j=0}^{J_i} \phi_{i}(s+j), 
\text{ where } \phi_i(s) \left| \frac{\tau_0(s)}{\hat{\tau}(s)}\right., 
\phi_i(s+J_i) \left| \frac{\tau_r(s)}{\hat{\tau}(s)}\right., 
\\ \notag
& & \text{ where each } J_i 
\text{ is the maximum of such numbers for the function } \phi_i, 
\\ \notag
& & \text{ and the outer product is over all such } \phi_i,
\\ 
\lbl{eq.Rs}
R(s)&=& \lcm_{0 \le i \le m}\left\{
\frac{b_i(s-i)Q(s)}{\gcd\left(b_i(s-i)Q(s), 
a_i(s-i)\prod_{j \neq i}b_j(s-i)Q(s+j-i) \right)}\right\}
\end{eqnarray} 
Obviously, $R(s), Q(s) \in K[s]$ are polynomials.

\begin{theorem}
\lbl{thm.multigosper}
With the above conventions, every rational solution of \eqref{eq.abx} 
has the form
$$
x(s) = \frac{R(s) P(s)}{Q(s)},
$$
where $P(s) \in K[s]$ is a polynomial.
\end{theorem}

\begin{proof}
Suppose $x(s) = \frac{A(s)}{B(s)}$, with $\gcd(A(s), B(s))=1$. Then 
$$
\sum_{i=0}^{m}\frac{a_i(s)\sigma(s)}{b_i(s)}\frac{A(s+i)}{B(s+i)} 
= c(s)\sigma(s).
$$

So
$$
\sum_{i=0}^{m}\frac{a_i(s)}{B(s+i)}\frac{A(s+i)\sigma(s)}{b_i(s)} 
= c(s)\sigma(s),
$$
$$
c(s)\sigma(s)\prod_{j=0}^{m}B(s+j) 
= \sum_{i=0}^{m}A(s+i)\tau_i(s+i)\prod_{j \neq i}B(s+j).
$$
Since $\tau_i(s)$ are polynomials for all $i$, it follows that 
$$
B(s+i) \left| A(s+i)\tau_i(s+i)\prod_{j \neq i}B(s+j)\right. .
$$
Since $\gcd(A(s), B(s))=1$, it follows that
$$
B(s) \left| \tau_i(s)\prod_{j \neq i}^{m}B(s+j-i)\right. .
$$ 
Write $B(s) = \prod_{i=0}^{U}\prod_{j=0}^{V_i}f_i(s+j)\prod_{j=0}^W g_j(s)$, 
where $U, V_i, W$ are constants; and $\gcd(g_i(s), g_j(s+L)) = 1$ for any 
$i, j, L$; and $\gcd(g_j(s), f_i(s+l)) = 1$ for any $i, j$, and 
$-m \le l \le m$. We call the functions $g_j(s)$ {\em singletons}, 
and $\{f_i(s+j)\}_{0 \le j \le V_j}$ {\em chains}, in which $f_i(s)$ are the 
{\em heads} of chains, and $f_i(s+V_i)$ the {\em tails} of chains. 
So we are writing $B(s)$ uniquely as a product of chains and singletons.

There are two cases:

Case I: The tail of one chains is always far apart from the head of 
another in $B(s)$, 
i.e., $\gcd(f_i(s+V_i), f_j(s+v)) = 1$ for all $0 \le i,j \le U$ and 
$-m \le v \le m$. Then
 
\begin{eqnarray*}
\left.\prod_{i=0}^{U}f_i(s)\prod_{j=0}^W g_j(s) 
= \frac{B(s)}{\gcd(B(s), \prod_{j=1}^mB(s+j))} \right| 
&&\gcd(B(s), \tau_0(s)),
\\
\left.\prod_{i=0}^{U}f_i(s+V_i)\prod_{j=0}^W g_j(s) 
= \frac{B(s)}{\gcd(B(s), \prod_{j=1}^mB(s-j))} \right| 
&&\gcd(B(s), \tau_m(s)),	
\\
\left.\prod_{j=0}^W g_j(s) = \frac{B(s)}{\gcd(B(s), 
\prod_{j \neq i}B(s+j-i))} \right| &&\gcd(B(s), \tau_i(s)), i \neq 0, m.
\end{eqnarray*}
Thus the singletons have the property

$$
\left.\prod_{j=0}^W g_j(s) \right| \gcd\left(\tau_i, 0 \le i \le m\right).
$$

At the same time, the heads of the chains $f_i(s)$ in $B(s)$ are factors of  
$\tau_0$, and the tails $f_i(s+V_i)$ factors of $\tau_r$. 
Therefore each chain in $B(s)$ factors $\prod_{j=0}^{J_\ell} 
\phi_{\ell}(s+j)$ for some $\ell$. Recalling the definition of $Q(s)$ from
Equation \eqref{eq.Qs}, it follows that 
$B(s)$ divides $Q(s)$.

Case II: The heads and tails of chains are close, 
i.e., $\gcd(f_i(s+V_i), f_j(s+v)) \neq 1$ for some $0 \le i,j \le U$ and 
$-m \le v \le m$. 
In this case, $\prod_{i=0}^{I} \prod_{j=0}^{J_i} \phi_{i}(s+j)$ will 
contain a chain whose head is 
$f_i(s)$ and tail is $f_j(s+V_j)$ in $Q(s)$. This is a longer chain than 
what $B(s)$ really needs, but it still guarantees that $B(s)$ divides $Q(s)$.

So far, this proves that $x(s)= \frac{A(s)}{Q(s)}$ where $A(s) \in K[s]$
is a polynomial. To finish the proof, it suffices to show that $R(s)$ (given
by Equation \eqref{eq.Rs}) divides $A(s)$. 
Since 
$$
\sum_{i=0}^m \frac{a_i(s)}{b_i(s)}\frac{A(s+i)}{Q(s+i)} = c(s),
$$
with $a_i, b_i, c$ polynomials, any polynomial factor that appears only 
once in the $m+1$ denominators on the left-hand side must also divide the 
corresponding numerator, which means
$$
\left. \frac{b_i(s)Q(s+i)}{\gcd\left(b_i(s)Q(s+i), 
a_i(s)\prod_{j \neq i}b_j(s)\prod_{j \neq i}Q(s+j)\right)} 
\right| A(s+i).
$$
\end{proof}

\begin{remark}
\lbl{rem.gosper}
When $m=1$, Theorem~\ref{thm.multigosper}
becomes Gosper's algorithm. Recall that Gosper's algorithm 
tries to find rational solution $x(s)$ such that
$$
\frac{a(s)c(s+\ell)}{b(s)c(s)} x(s+1) - x(s) = 1
$$
for some integer $\ell$. Based on our propositions, we get a chain 
$\prod_{i=0}^{\ell-1}c(s+i)$ as the denominator and $b(s-1)$ as part of the 
numerator, which agrees with Gosper's result.
\end{remark}

\begin{remark}
\lbl{rem.abramov}
Abramov's universal denominator treats the singletons in 
Theorem~\ref{thm.multigosper} as chains of length 1, 
and then tries to find all chains.
However, by picking singletons out first, 
we reduce the possibility of generating redundant chains in the denominator,
because factors in the leading coefficient may mingle with the singletons
and generate unwanted factors in chains.
We illustrate the effect by example.
\end{remark}

\begin{example}
This is Example 1 in \cite{Ab}.
\begin{eqnarray*}
(n+4)(2n+1)(n+2)x(n+3)-(2n+3)(n+3)(n+1)x(n+2)+n(n+2)(2n-3)x(n+1)\\
-(n-1)(2n-1)(n+1)x(n)=0.
\end{eqnarray*}
Abramov's algorithm gives the denominator $u(n)=n^3-n$ for all rational 
function solutions $x(n) \in \BQ(n)$ of the above equation, and computes 
the general polynomial solution $C(2n^2-3n)$.
However our algorithm finds two singletons $(n+1)(n-1)$ and no chains. 
So the denominator is $Q(n)=n^2-1$, which strictly divides $u(n)$. 
\end{example}

\begin{example} 
In one of the intermediate steps for Example~\ref{ex.3}, we get 
{\small
\begin{eqnarray*}
x(n,k_1,k_2) & + &  \frac{-(2k_2^2+k_2+4k_2k_1-6k_2n-3n+k_1+3n^2-6k_1n+2k_1^2)
(n+k_2+2)(-n+k_2+1)}{(k_2+2)^2(k_1+1-n+k_2)(k_1-3n+k_2)}x(n,k_1,k_2+1) 
\\
& + & \frac{(k_1+1-n+k_2)^2(n+k_2+3)(n+k_2+2)(-n+k_2+2)(-n+k_2+1)}{
(k_2+3)^2(k_2+2)^2(k_1+2-n+k_2)(k_1-3n+k_2+1)}x(n,k_1,k_2+2)
\\
& = &
\frac{c(n,k_1,k_2)}{(n+k_2+1)(-n+k_2)\prod_{j=0}^{2}
\left[(k_1-3n+k_2+j)(k_1-n+k_2+1+j)(k_1+1+j)^2\right]},
\end{eqnarray*}
}
to solve for $x(n,k_1,k_2)$ with $c(n,k_1,k_2)$ a polynomial. 
After multiplying both sides by the denominator of the right-hand side, 
we find four singletons $(n+k_2+1)(-n+k_2)(k_1-3n+k_2)(k_1-n+k_2+1)$; no chain 
in the denominator of $x(n,k_1,k_2)$; 
and $k_2^2(k_2+1)^2$ as factors of the numerator of $x(n,k_1,k_2)$. Hence 

$$
x(n,k_1,k_2) = 
\frac{k_2^2(k_2+1)^2}{(n+k_2+1)(-n+k_2)(k_1-3n+k_2)(k_1-n+k_2+1)}P(n,k_1,k_2),
$$
where $P(n,k_1,k_2)$ is a polynomial.
\end{example}

Our method keeps finding the best possible denominators in 
all the steps of the examples discussed in the paper.

\ifx\undefined\bysame
        \newcommand{\bysame}{\leavevmode\hbox
to3em{\hrulefill}\,}
\fi

\end{document}